\documentclass[a4paper, 11pt]{article}
\usepackage{amsmath,amssymb,amsthm,mathrsfs,comment}
\usepackage[top=30truemm,bottom=30truemm,left=30truemm,right=30truemm]{geometry}
\usepackage[pdftex]{graphicx}
\usepackage{tikz}
\usepackage{mathtools}
\usepackage{here}

\usetikzlibrary{shapes,arrows,positioning,intersections,calc}
\usetikzlibrary{shapes.misc}
\usetikzlibrary{decorations.markings}
\tikzset{cross/.style={cross out, draw=black, minimum size=2*(#1-\pgflinewidth), inner sep=0pt, outer sep=0pt},
cross/.default={1pt}}
\tikzset{->-/.style={decoration={
  markings,
  mark=at position #1 with {\arrow[scale=1.5]{>}}},postaction={decorate}}}

\tikzset{-<-/.style={decoration={
  markings,
  mark=at position #1 with {\arrow[scale=1.5]{<}}},postaction={decorate}}}

\tikzset{midstealth/.style={decoration={
  markings,
  mark=at position #1 with {\arrow{stealth}}},postaction={decorate}}}
\newtheoremstyle{break}
  {}%
  {}%
  {\itshape}
  {}%
  {\bfseries}
  {.}%
  {\newline}%
  {}%

\theoremstyle{plain}
\newtheorem{thm}{Theorem}[section]
\newtheorem{rem}[thm]{Remark}
\newtheorem{exa}[thm]{Example}

\newtheorem{prop}[thm]{Proposition}
\newtheorem{defn}[thm]{Definition}

\newcommand{\Z}{\mathbb{Z}}
\newcommand{\R}{\mathbb{R}}
\newcommand{\C}{\mathbb{C}}
\newcommand{\PP}{\mathbb{P}}
\newcommand{\ii}{\sqrt{-1}}
\DeclareMathOperator{\Homo}{H}
\DeclareMathOperator{\codim}{codim}

\title{A tree expansion formula of a homology intersection numbers on the configuration space $\mathcal{M}_{0,n}$}
\author{Saiei-Jaeyeong Matsubara-Heo\footnote{Graduate School of Science, Kobe  University, 1-1 Rokkodai, Nada-ku, Kobe 657-8501, Japan.\newline e-mail: \texttt{saiei@math.kobe-u.ac.jp}}}

\begin{document}

\date{}
\maketitle

\begin{abstract}      
In \cite{M}, Sebastian Mizera discovered a tree expansion formula of a homology intersection number on the configuration space $\mathcal{M}_{0,n}$. The formula originates in a study of Kawai-Lewellen-Tye relation in string theory. In this paper, we give an elementary proof of the formula. The basic ingredients are the combinatorics of the real moduli space $\overline{\mathcal{M}}_{0,n}(\R)$ and a combinatorial identity related to the face number of the associahedron. 
\end{abstract}

\section{Introduction}
One of the distinguished properties of hypergeometric functions is that they enjoy integral representations whose integrands are elementary functions. These integrals are called {\it hypergeometric integrals}. The standard machinery of studying the hypergeometric integral is to regard it as a period integral, a pairing between a de Rham cohomology group and a homology group with local system coefficient. The latter object is also called {\it twisted homology group} in the context of special functions. Combining this viewpoint with Poincar\'e duality, the intersection theory of twisted homology groups naturally comes into play (\cite{G}, \cite{GMH}, \cite{MW}, \cite{MH}, \cite{MaY}, \cite{Mim}, \cite{MY}, \cite{M}, \cite{OMY}, \cite{OST}).

Though a period integral can be defined for a more general class of integrals, a remarkable feature of the intersection theory of hypergeometric integral is that intersection numbers of specific homology classes have a combinatorial expression. In this paper, we are interested in the intersection number associated to the following integral:
\begin{equation}\label{eqn:0}
\int_{\Delta}\prod_{1\leq i<j\leq n}(t_i-t_j)^{s_{ij}}\omega.
\end{equation}
Here, the ambient space of integration is the configuration space $\mathcal{M}_{0,n}:={\rm Conf}_{n}(\PP^1)/{\rm GL}(2;\C)$ of $n$-points on the Riemann sphere $\PP^1$ with ${\rm Conf}_{n}(\PP^1):=\{ (t_1,\dots,t_n)\in(\PP^1)^n\mid t_i\neq t_j\text{ for any }1\leq i<j\leq n\}$, $s_{ij}$ are complex parameters subject to a linear constraint, $\Delta$ is an integration cycle and $\omega$ is a meromorphic differential form on $\mathcal{M}_{0,n}$. Note that the action of ${\rm GL}(2;\C)$ on ${\rm Conf}_{n}(\PP^1)$ is given by m\"obius transform of each coordinate $t_i$. If we normalize $t_1$, $t_{n-1}$ and $t_n$ to $0,1$ and $\infty$, $\mathcal{M}_{0,n}$ is the complement of the Selberg arrangement. 

The real part $\mathcal{M}_{0,n}(\R)$ of the configuration space $\mathcal{M}_{0,n}$ consists of finitely many connected components $\{ \Delta(\alpha)\}_\alpha$ labeled by the quotient $\mathfrak{S}_n/D_n$ of the permutation group $\mathfrak{S}_n$. Here, $D_n$ is a subgroup of $\mathfrak{S}_n$ isomorphic to the dihedral group of order $2n$. We can regard $\Delta(\alpha)$ as an element of the twisted (Borel-Moore) homology group\footnote{In this paper, we assume that $s_{ij}$ are generic parameters in a sense that will be clarified later (see \S\ref{sec:4}). Under this condition, the twisted homology group is canonically isomorphic to its Borel-Moore counterpart.} by specifying the determination of the multivalued function $\prod_{1\leq i<j\leq n}(t_i-t_j)^{s_{ij}}$ and we write $[C^+(\alpha)]$ for the homology class determined by $\Delta(\alpha)$ in this sense. The set $\{ [C^+(\alpha)]\}_\alpha$ generates the twisted homology group. If we write $[C^-(\alpha)]$ for the homology class obtained by replacing $s_{ij}$ by $-s_{ij}$ in the definition of $[C^+(\alpha)]$, it is natural to expect that the homology intersection number $\langle [C^+(\alpha)],[C^-(\beta)]\rangle_h$ has a combinatorial formula. The study of the homology intersection number $\langle [C^+(\alpha)],[C^-(\beta)]\rangle_h$ is not new. For example, the authors of \cite{OMY} describes the recursive structure of the intersection number with respect to the natural fibration $\mathcal{M}_{0,n}\rightarrow\mathcal{M}_{0,n-1}$. The usual Selberg integral (\cite{Sel}) appears if we specialize the parameters $s_{ij}$ of (\ref{eqn:0}) in a specific manner. In this case, a symmetric group acts on the twisted homology group and the invariant part of the (dual) twisted homology group is a one-dimensional vector space spanned by the class $[C^\pm]:=\sum_{\alpha\in\mathfrak{S}_{n-3}}[C^\pm(\alpha)]$. Here, the subset $\mathfrak{S}_{n-3}\subset \mathfrak{S}_{n}/D_n$ of permutations of $\{ 1,\dots,n\}$ whose element fixes $1,n-1$ and $n$ corresponds to {\it the bounded chambers} of the Selberg arrangement. The authors of \cite{MY} evaluated the self intersection number $\langle [C^+],[C^-]\rangle_h$ in terms of sine functions of the parameters.

In \cite{M}, Sebastian Mizera discovered yet another formula of the intersection number $\langle [C^+(\alpha)],[C^-(\beta)]\rangle_h$: a tree expansion formula. Tree diagrams naturally appear from the fact that the closure $K(\alpha)$ of each $\Delta(\alpha)$ in the real part $\overline{\mathcal{M}}_{0,n}(\R)$ of the Deligne-Knudsen-Mumford compactification $\overline{\mathcal{M}}_{0,n}$\footnote{In \cite{M}, the Deligne-Knudsen-Mumford compactification $\overline{\mathcal{M}}_{0,n}$ is denoted by $\widetilde{\mathcal{M}}_{0,n}$} (\cite{Knud}) is the associahedron (\cite{S}), whose faces are in one-to-one correspondence to a set of trees. What is remarkable in his formula is the fact that only a few tree diagrams actually contribute to the homology intersection number $\langle C^+(\alpha),C^-(\beta)\rangle_h$. Namely, we only need to focus on the trees of which the valency at any internal vertex is  odd. We call such a tree diagram an {\it admissible tree}. Following \cite{M}, we set $\langle C^+(\alpha),C^-(\beta)\rangle_h=\left(\frac{\ii}{2}\right)^{n-3}m(\alpha|\beta)$\footnote{More precisely, $m(\alpha|\beta)$ should be denoted by $m_1(\alpha|\beta)$ (\cite{M}, \cite{MI}).}. Let us explain how the formula looks like when both $\alpha$ and $\beta$ are taken to be the identity permutation $\mathbb{I}_n:=12\cdots n$. If $T$ is an admissible tree, we assign a Catalan number $C_{\frac{|v|-3}{2}}$ to each internal vertex $v$ and assign a linear combination $s_e$ of the parameters $s_{ij}$ to each internal edge $e$. Then, the formula of \cite[Lemma 4.1]{M} computes the number $m(\mathbb{I}_n|\mathbb{I}_n)$ as a sum of terms of the form 
\begin{equation}\label{eqn:Intro}
\prod_{v:internal\ vertex\ of\ T}C_{\frac{|v|-3}{2}}\prod_{e:internal\ edge\ of\ T}\cot(\pi s_e)
\end{equation}
for all admissible trees $T$. A precise formulation is given in \S\ref{sec:4} of this paper.

Interestingly, the tree expansion formula was discovered in physical context. The Kawai-Lewellen-Tye relation in string theory is a manifestation of open/closed string duality and it expresses the closed string amplitude in terms of a quadratic combination of open string amplitudes. We can arrange the coefficients of the quadratic combination into a square matrix which we call the KLT kernel. It turns out that, at tree-level, KLT relation can be reformulated as the {\it twisted period relation} (\cite[THEOREM 2]{CM}, \cite[(5.1)]{HY}, \cite[Theorem 6.2]{MH2}), a relation among period integrals, twisted homology and cohomology intersection numbers (\cite[(3.19)]{M}). Thus, the inverse of the KLT kernel can be seen as the twisted homology intersection matrix. The tree expansion formula of $m(\alpha|\beta)$ in \cite[Lemma 4.1, Theorem 4.1]{M} is then an {\it $\alpha^\prime$-correction} of the formula of the field theory limit of the inverse of the KLT kernel  obtained in \cite{CHY}. More precisely, if we replace the parameters $s_{ij}$ by rescaled ones $\alpha^\prime s_{ij}$, the limit $\alpha^\prime\rightarrow 0$ of (\ref{eqn:Intro}) multiplied by $(\pi\alpha^\prime)^{n-3}$ computes the field theory inverse KLT kernel (\cite[(3.2)]{MI}). At the limit, we only have the contribution from all trivalent tree diagrams which amounts to focusing on the vertices of the associahedron.




Unfortunately (for mathematicians), the argument in \cite{M} makes use of physical intuition developed in \cite{MI} which deals with the $\alpha^\prime$-correction of the bi-adjoint scalar amplitude discussed in \cite{CHY}. The aim of this paper is to provide a math-friendly proof of the tree expansion formula of \cite[Lemma 4.1, Theorem 4.1]{M}. Once we recall the well-known cell decomposition of $\overline{\mathcal{M}}_{0,n}(\R)$ (\cite{Devadoss}), it is easy to see that the homology intersection number $\langle C^+(\alpha),C^-(\beta)\rangle_h$ has a tree expansion formula. Indeed, it has already appeared in \cite[Lemma 1]{MY} in a slightly different form. However, the sum is taken over all tree diagrams and we have sine-like functions in the summand. Therefore, it is important to see why many terms cancel each other in the cotangent expansion. The answer is a simple, but non-trivial combinatorial identity (Theorem \ref{prop:WZ}). Combining Theorem \ref{prop:WZ} with the combinatorics of the real moduli space $\overline{\mathcal{M}}_{0,n}(\R)$, we obtain the desired formula. Since the proof turned out to be short and concise, we expect an analogous formula of homology intersection numbers of Coxeter arrangements (\cite{ACDELM}). This aspect will be discussed in a forthcoming paper.



The author thanks Sebastian Mizera for letting me know his formula \cite[Theorem 4.1]{M}, asking me if there is a short mathematical proof of the result and many other valuable comments. The author also thanks Genki Shibukawa for reporting to me a simpler proof of Theorem \ref{prop:WZ} than the one we originally obtained. With his permission, we include his proof as the second proof of Theorem \ref{prop:WZ}. This work is supported by JSPS KAKENHI Grant Number 19K14554 and JST AIP-PRISM Grant.

\section{Convention for faces of associahedron}

\indent
In this section, we introduce some basic notation related to the associahedron. Let $n\geq 3$ be an integer and let us consider a sequence of letters $12\cdots n-1$. A bracket $a$ is a consecutive digits $a=i\cdots j$ ($1\leq i< j\leq n-1$) which is not $12\cdots n-1$. We write $|a|$ for the length $j-i+1$ of $a$. A bracketing $F$ of letters $12\cdots n-1$ is a collection of brackets such that for any pair of elements $a,a^\prime\in F$ either $a\cap a^\prime=\varnothing$, $a\subset a^\prime$, or $a^\prime \subset a$ is true.
\begin{defn}[associahedron]
The associahedron (or Stasheff polytope) $K_{n-1}$ is a convex polytope of dimension $n-3$ whose face poset is isomorphic to that of bracketings of $n-1$ letters $12\cdots n-1$, ordered so that $F\prec F^\prime$ if $F$ is obtained from $F^\prime$ by adding new brackets. 
\end{defn}

If $F$ is a face of associahedron $K_{n-1}$, we write $F<K_{n-1}$. A face $F<K_{n-1}$ corresponds to a polyhedral subdivision of a planer convex $n$-gon whose edges are labeled by $1,2,\dots, n$ in a clockwise order on  a unit circle which we regard as the boundary of a unit disk.  We define $\codim F$ as the number of brackets appearing in $F$. We set $\dim F:=n-3-\codim F$. Another interpretation of a face $F<K_{n-1}$ is a rooted tree embedded in a unit disk of which the external vertices are labeled by $1,2,\dots, n$ in a clockwise manner. The label $n$ corresponds to the root vertex. The rule is as follows:
\begin{enumerate}
\item Write a planer $n$-gon whose edges are labeled by $1,\dots,n$ in a clockwise way.
\item To each polygon $\Delta_i$ in $F$, we associate a vertex $v_i$ located at the barycenter of $\Delta_i$. Draw external edges from $v_i$ to edges of $\Delta_i$. Finally, if $\Delta_i$ and $\Delta_j$ share an edge, we connect $v_i$ and $v_j$.
\end{enumerate}

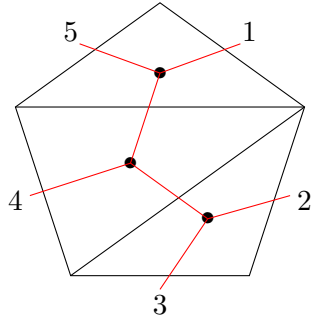
\begin{figure}[H]
\begin{center}
\begin{tikzpicture}
\coordinate (V3) at (90:2);
\coordinate (V4) at (18:2);
\coordinate (V1) at (-54:2);
\coordinate (V2) at (-126:2);
\coordinate (V5) at (162:2);
\node at (54:2){$1$};
\node at (-18:2){$2$};
\node at (-90:2){$3$};
\node at (126:2){$5$};
\node at (198:2){$4$};
\coordinate (v1) at ({0.66*(cos(18)+cos(-54)+cos(-126))},{0.66*(sin(18)+sin(-54)+sin(-126))});
\coordinate (v2) at ({0.66*(cos(162)+cos(-126)+cos(18))},{0.66*(sin(162)+sin(-126)+sin(18))});
\coordinate (v3) at ({0.66*(cos(90)+cos(162)+cos(18))},{0.66*(sin(90)+sin(162)+sin(18))});
\node at (v1){$\bullet$};
\node at (v2){$\bullet$};
\node at (v3){$\bullet$};
\draw (V3) -- (V4) -- (V1) -- (V2) -- (V5) -- cycle;
\draw (V2) -- (V4);
\draw (V4) -- (V5);
\draw[red] (v1) -- (v2) -- (v3);
\draw[red] (v1) -- (-18:1.8);
\draw[red] (v1) -- (-90:1.8);
\draw[red] (v2) -- (198:1.8);
\draw[red] (v3) -- (54:1.8);
\draw[red] (v3) -- (126:1.8);
\end{tikzpicture}
\caption{The tree and the polyhedral subdivision corresponding to $1((23)4)$}
\label{TheTree}
\end{center}
\end{figure}

By abuse of notation, the resulting graph is still denoted by $F$. An unlabeled vertex $v$ of $F$ is called an internal vertex and the symbol $V(F)_{int}$ denotes the set of internal vertices. An edge $e=(v_1,v_2)$ of $F$ is called an internal edge if both $v_1$ and $v_2$ are internal vertices. The set of internal edges of $F$ is denoted by $E(F)_{int}$. The valency (or degree) $|v|$ of a vertex $v$ is the number of edges containing $v$. If we take any internal edge $e$ from $F$, $F$ is decomposed into a pair of connected components. The component which does not contain the external vertex $n$ defines a subset of labels $a\subset\{ 1,\dots,n-1\}$ consisting of consecutive numbers. This $a$ corresponds to a bracket in $1\cdots(n-1)$. Conversely, any bracket $a\in F$ is obtained from an internal edge in this fashion. 

Trees are identified with each other under the dihedral symmetry. With this in mind, we can also view the tree diagram $F$ as a bracketing of $(i+1)\cdots n1\cdots (i-1)$ for any $i=1,\dots,n-1$. For example, the tree in Figure \ref{TheTree} can be identified with $1((23)4)=((23)4)5=3(4(51))=(4(51))2=(51)(23)$. The connected component of the complement of an internal edge $e$ which does not contain the external vertex $i$ corresponds to a bracket in $(i+1)\cdots n1\cdots (i-1)$. If we regard $F$ as a bracketing of $(i+1)\cdots n1\cdots (i-1)$, any bracket is obtained from an internal edge in this fashion. In the following, we regard a face $F<K_{n-1}$ as a bracketing of the digits $1\cdots(n-1)$ unless otherwise stated.

\begin{defn}
A face $F$ is said to be admissible if for each internal vertex $v$, the valency $|v|$ is odd.
\end{defn}

We conclude this section by recalling a well-known
\begin{prop}[\S2 of \cite{S}]\label{prop:2.3}
Any face $F<K_{n-1}$ is isomorphic to a product of associahedra. To be more precise, one has an isomorphism
\begin{equation}\label{eqn:face}
F\simeq \prod_{v\in V(F)_{\rm int}}K_{|v|-1}.
\end{equation}
\end{prop}

\noindent
For readers' convenience, let us explain the meaning of the isomorphism (\ref{eqn:face}). If $F^\prime<F$ is a face, $F^\prime$ defines a polyhedral subdivision of a planer convex $n$-gon which is a refinement of $F$. Therefore, $F^\prime$ induces a polyhedral subdivision of each polygon appearing in $F$. It is easy to see that the valency $|v|$ at a vertex $v\in V(F)_{int}$ is equal to the number of edges of the polygon containing $v$. Thus, the proposition follows.

\section{Real moduli space $\overline{\mathcal{M}}_{0,n}(\R)$ as a patchwork of associahedra (\cite{Devadoss}, \cite{Kap}, \cite{YoshidaStasheff})}
In this section, we briefly recall the combinatorics of the real part $\overline{\mathcal{M}}_{0,n}(\R)$ of the moduli space $\overline{\mathcal{M}}_{0,n}$ of stable pointed curves of genus $0$. The readers can refer to \cite{Devadoss}, \cite{Kap} or \cite{YoshidaStasheff} for proofs and more explanations.

Let $\mathbb{P}^1$ denote the complex projective line. We set ${\rm Conf}_{n}(\mathbb{P}^1):=\{ (t_1,\dots,t_n)\in (\mathbb{P}^1)^n\mid t_i\neq t_j (1\leq i<j\leq n)\}$. The group ${\rm GL}(2;\C)$ acts on ${\rm Conf}_{n}(\mathbb{P}^1)$ through m\"obius transform of each coordinate $t_i$. The quotient ${\rm Conf}_{n}(\mathbb{P}^1)/{\rm GL}(2;\C)$ is denoted by $\mathcal{M}_{0,n}$. Moving $t_1,t_{n-1}$ and $t_n$ to $0,1$ and $\infty$, we have an identification $\mathcal{M}_{0,n}\simeq \{ (t_2,\dots,t_{n-2})\in\C^{n-3}\}\setminus\bigcup_{i=2}^{n-2}\{ t_i(t_i-1)=0\}\cup\bigcup_{2\leq i<j\leq n-2}\{ t_i=t_j\}$. We set $\Delta:=\{ 0<t_2<\cdots<t_{n-2}<1\}\subset\mathcal{M}_{0,n}$. The permutation group $\mathfrak{S}_n$ of $\{ 1,\dots,n\}$ naturally acts on $\mathcal{M}_{0,n}$ by $\alpha\cdot (t_1,\dots,t_n):=(t_{\alpha^{-1}(1)},\dots,t_{\alpha^{-1}(n)})$. For any element $\alpha\in\mathfrak{S}_n$, we set $\Delta(\alpha):=\alpha\cdot\Delta$. These chambers $\Delta(\alpha)$ cellulate the real part of $\mathcal{M}_{0,n}$.

The Deligne-Knudsen-Mumford compactification $\overline{\mathcal{M}}_{0,n}$ of $\mathcal{M}_{0,n}$ is a smooth projective variety defined over $\mathbb{Q}$ (\cite{Knud}). In this paper, we simply regard it as a complex variety. The complex structure determines the set of real points $\overline{\mathcal{M}}_{0,n}(\R)$, which was investigated in detail in \cite{Devadoss}. For our purpose, it is important to recall the cell decomposition of $\overline{\mathcal{M}}_{0,n}(\R)$ as a patchwork of associahedra. Any element $\alpha\in\mathfrak{S}_n$ is a bijection of the set $\{ 1,\dots,n\}$ and we identify $\alpha$ with the number sequence $\alpha(1)\alpha(2)\cdots\alpha(n)$. In this sense, let $D_n$ be a subgroup of $\mathfrak{S}_n$ generated by two elements $23\cdots n1$ and $n (n-1)\cdots 1$. It is easy to see that $D_n$ is isomorphic to the dihedral group of order $2n$. Moreover, the natural inclusion $\mathfrak{S}_{n-1}\hookrightarrow\mathfrak{S}_{n}$ induces an isomorphism $\mathfrak{S}_{n-1}/\langle (n-1) (n-2) \cdots 1\rangle\simeq\mathfrak{S}_{n}/D_n$. The closure of the cell $\Delta(\alpha)$ in $\overline{\mathcal{M}}_{0,n}(\R)$ gives rise to an associahedron for which we write $K(\alpha)$. Let us choose a representative $[\alpha]\in\mathfrak{S}_n/D_n$ so that $\alpha(n)=n$. Then, the set of brackets on $\alpha(1)\cdots\alpha(n-1)$ forms an associahedron which we identify with $K(\alpha)$. $K(\alpha)$ and $K(\beta)$ are identified in $\overline{\mathcal{M}}_{0,n}(\R)$ precisely when the equivalence classes $[\alpha]$ and $[\beta]$ are identical in the quotient $\mathfrak{S}_{n}/D_n$. On the other hand, any face $F_1<K(\alpha)$ is identified with a face $F_2<K(\beta)$ in $\overline{\mathcal{M}}_{0,n}(\R)$ precisely when the corresponding polyhedral subdivisions of a planer convex $n$-gon are related to each other by a sequence of {\it twists along diagonals} in the sense of \cite[\S3.1]{Devadoss}. If we regard $F_1$ and $F_2$ as labeled trees, they are identified in $\overline{\mathcal{M}}_{0,n}(\R)$ precisely when one is obtained by the other by a sequence of twists around an internal edge (Figure \ref{Isomers}). In particular, twists do not change the set of internal/external vertices nor the set of internal/external edges.

\begin{figure}[H]
\begin{minipage}{0.3\hsize}
\begin{center}
\begin{tikzpicture}
\coordinate (V3) at (90:2);
\coordinate (V4) at (18:2);
\coordinate (V1) at (-54:2);
\coordinate (V2) at (-126:2);
\coordinate (V5) at (162:2);
\node at (54:2){$1$};
\node at (-18:2){$2$};
\node at (-90:2){$3$};
\node at (126:2){$5$};
\node at (198:2){$4$};
\coordinate (v1) at ({0.66*(cos(18)+cos(-54)+cos(-126))},{0.66*(sin(18)+sin(-54)+sin(-126))});
\coordinate (v2) at ({0.66*(cos(162)+cos(-126)+cos(18))},{0.66*(sin(162)+sin(-126)+sin(18))});
\coordinate (v3) at ({0.66*(cos(90)+cos(162)+cos(18))},{0.66*(sin(90)+sin(162)+sin(18))});
\node at (v1){$\bullet$};
\node at (v2){$\bullet$};
\node at (v3){$\bullet$};
\draw (V3) -- (V4) -- (V1) -- (V2) -- (V5) -- cycle;
\draw (V2) -- (V4);
\draw (V4) -- (V5);
\draw[red] (v1) -- (v2) -- (v3);
\draw[red] (v1) -- (-18:1.8);
\draw[red] (v1) -- (-90:1.8);
\draw[red] (v2) -- (198:1.8);
\draw[red] (v3) -- (54:1.8);
\draw[red] (v3) -- (126:1.8);
\end{tikzpicture}
\end{center}
\end{minipage}
\begin{minipage}{0.3\hsize}
\begin{center}
\begin{tikzpicture}
\coordinate (V3) at (90:2);
\coordinate (V4) at (18:2);
\coordinate (V1) at (-54:2);
\coordinate (V2) at (-126:2);
\coordinate (V5) at (162:2);
\path (V3) ++ (0,0.2) coordinate (VV3);
\path (V4) ++ (0,0.2) coordinate (VV4);
\path (V5) ++ (0,0.2) coordinate (VV5);
\path (V3) ++ (0,0.8) coordinate (A);
\path (V3) ++ (0,-1.6) coordinate (B);
\path (V3) ++ (0,0.4) coordinate (C);
\node at (54:2){$1$};
\node at (-18:2){$2$};
\node at (-90:2){$3$};
\node at (126:2){$5$};
\node at (198:2){$4$};
\node at (C){\rotatebox{180}{$\curvearrowright$}};
\coordinate (v1) at ({0.66*(cos(18)+cos(-54)+cos(-126))},{0.66*(sin(18)+sin(-54)+sin(-126))});
\coordinate (v2) at ({0.66*(cos(162)+cos(-126)+cos(18))},{0.66*(sin(162)+sin(-126)+sin(18))});
\coordinate (v3) at ({0.66*(cos(90)+cos(162)+cos(18))},{0.66*(sin(90)+sin(162)+sin(18))});
\node at (v1){$\bullet$};
\node at (v2){$\bullet$};
\node at (v3){$\bullet$};
\draw[dashed] (A) -- (B);
\draw (V4) -- (V1) -- (V2) -- (V5) -- cycle;
\draw (VV3) -- (VV4) -- (VV5) -- cycle;
\draw (V2) -- (V4);
\draw (V4) -- (V5);
\draw[red] (v1) -- (v2) -- (v3);
\draw[red] (v1) -- (-18:1.8);
\draw[red] (v1) -- (-90:1.8);
\draw[red] (v2) -- (198:1.8);
\draw[red] (v3) -- (54:1.8);
\draw[red] (v3) -- (126:1.8);
\end{tikzpicture}
\end{center}
\end{minipage}
\begin{minipage}{0.3\hsize}
\begin{center}
\begin{tikzpicture}
\coordinate (V3) at (90:2);
\coordinate (V4) at (18:2);
\coordinate (V1) at (-54:2);
\coordinate (V2) at (-126:2);
\coordinate (V5) at (162:2);
\node at (54:2){$5$};
\node at (-18:2){$2$};
\node at (-90:2){$3$};
\node at (126:2){$1$};
\node at (198:2){$4$};
\coordinate (v1) at ({0.66*(cos(18)+cos(-54)+cos(-126))},{0.66*(sin(18)+sin(-54)+sin(-126))});
\coordinate (v2) at ({0.66*(cos(162)+cos(-126)+cos(18))},{0.66*(sin(162)+sin(-126)+sin(18))});
\coordinate (v3) at ({0.66*(cos(90)+cos(162)+cos(18))},{0.66*(sin(90)+sin(162)+sin(18))});
\node at (v1){$\bullet$};
\node at (v2){$\bullet$};
\node at (v3){$\bullet$};
\draw (V3) -- (V4) -- (V1) -- (V2) -- (V5) -- cycle;
\draw (V2) -- (V4);
\draw (V4) -- (V5);
\draw[red] (v1) -- (v2) -- (v3);
\draw[red] (v1) -- (-18:1.8);
\draw[red] (v1) -- (-90:1.8);
\draw[red] (v2) -- (198:1.8);
\draw[red] (v3) -- (54:1.8);
\draw[red] (v3) -- (126:1.8);
\end{tikzpicture}
\end{center}
\end{minipage}
\caption{A twist along a diagonal amounts to taking an {\it isomer}}
\label{Isomers}
\end{figure}
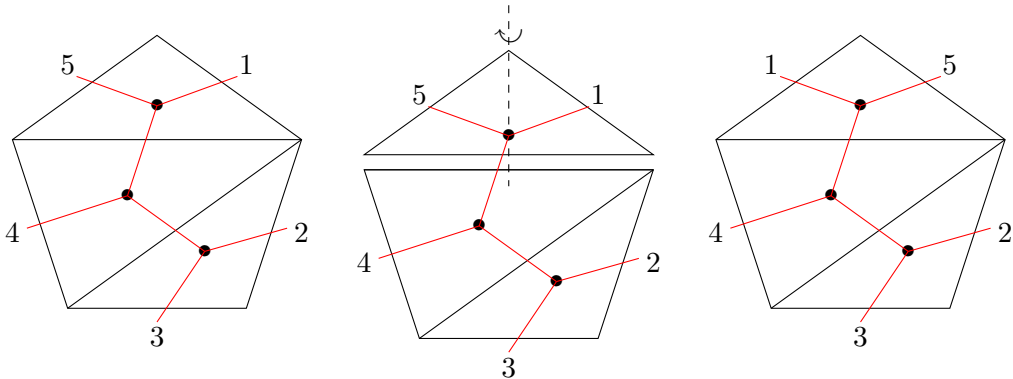

Here, it is an important question how we can compute the intersection $K(\alpha)\cap K(\beta)$ in $\overline{\mathcal{M}}_{0,n}(\R)$. The answer is quite simple. One can describe the intersection $K(\alpha)\cap K(\beta)$ in terms of a tree diagram. The graphical rule is simple and effective. We call this rule {\it CHY rule} since the same figure appeared in the paper \cite{CHY} by Cachazo, He and Yuan. CHY rule was further investigated in \cite{MI}. Let the symbol $\mathbb{I}_n$ denote the identity permutation $12\cdots n$. Since we have an identity $K(\alpha)\cap K(\beta)=\alpha\cdot (K(\mathbb{I}_n)\cap K(\alpha^{-1}\beta))$, we are reduced to the case of $K(\mathbb{I}_n)\cap K(\alpha)$. CHY rule is defined as follows:

\begin{enumerate}
\item Draw a circle with $n$ marked points $1,2,\dots,n$ arranged in a clock-wise order. Connect the labeled points $\alpha(i)$ and $\alpha(i+1)$ for any $i=1,\dots,n$ by a segment. Here, we put $\alpha(n+1):=\alpha(1)$. 

\item We have several polygons inside the circle. For any pair of polygons $\Delta$ and $\Delta^\prime$, we write $\Delta\sim\Delta^\prime$ if there is a sequence of polygons $\Delta=\Delta_0,\Delta_1,\dots,\Delta_k=\Delta^\prime$ such that $\Delta_i$ and $\Delta_{i+1}$ share a vertex and are in the diagonal position for $i=0,\dots,k-1$. Pick any polygon $\Delta$ with at least one marked point as a vertex. We write $\{ \Delta_1,\dots,\Delta_a\}$ for the set of polygons $\Delta^\prime$ such that $\Delta\sim\Delta^\prime$.

\item Associate a vertex $v_i$ to the barycenter of each $\Delta_i$. Connect each $v_i$ to the marked points in $\Delta_i$. Connect a pair of vertices $v_i$ and $v_j$ if $\Delta_i$ and $\Delta_{j}$ share a vertex and are in a diagonal position.
\end{enumerate}

Here, we assumed that the marked points $1,2,\dots,n$ are in a general position. Namely, they are arranged so that any triplet of segments connecting $\alpha(i)$ and $\alpha(i+1)$ does not have an intersection. The rule above produces a graph $G$. One may notice that there is an ambiguity in the third step and $G$ is not uniquely determined. Nonetheless, we have a

\begin{prop}\label{prop:CHY}
Let $G$ be the graph produced by CHY rule. The intersection $F=K(\mathbb{I}_n)\cap K(\alpha)$ is non-empty if and only if $G$ is a tree. If $G$ is a tree, $G$ is uniquely determined and it is the tree diagram corresponding to the face $F$ of $K(\mathbb{I}_n)$. 
\end{prop}

\begin{proof}
Recall that the equivalence class $[\alpha]\in\mathfrak{S}_n/D_n$ is uniquely determined by the intersection $F=K(\mathbb{I}_n)\cap K(\alpha)$ if it is non-empty. Therefore, it is enough to prove that if the intersection $F$ is non-empty, the graph produced by CHY rule is a tree and it coincides with the tree diagram corresponding to $F$. 

We regard $F$ as a collection of brackets in $12\cdots (n-1)$. We choose a maximal (with respect to inclusion) element $a\in F$. Let us consider a planer convex $n$-gon whose vertices are labeled by $1,2,\dots,n$ in a clockwise order. Since $a$ is a set of consecutive digits $a=i\cdots j$, we can flip the vertices $i\cdots j$ to obtain an {\it hourglass}\footnote{One can also regard the hourglass as a {\it bubble} (\cite{Devadoss}).}. We replace $F$ by $F\setminus \{ a\}$ and repeat this process. Each time we pick a maximal element $a\in F$, we flip the vertices contained in $a$ to obtain an hourglass with several sections. If we regard this hourglass as a tree diagram, it is the tree diagram corresponding to $F$. An example of $F=12(3(456)7)$ is illustrated in Figure \ref{Hourglass}. We can also recover the permutation $\alpha$ such that $K(\mathbb{I}_n)\cap K(\alpha)=F$ as follows. We begin with the consecutive digits $\alpha=12\cdots n$. Each time we take a maximal element $a=i\cdots j\in F$, we revert the digits $i\cdots j$ or $j\cdots i$ in $\alpha$. In the end, we arrive at a number sequence $\alpha(1)\cdots \alpha(n-1)\alpha(n)$. For example, if we take $F=12(3(456)7)$, the process is $12345678\rightarrow 12765438\rightarrow 12745638$. By construction, CHY rule applied to the permutation $\alpha$ produces the tree diagram corresponding to $F$.

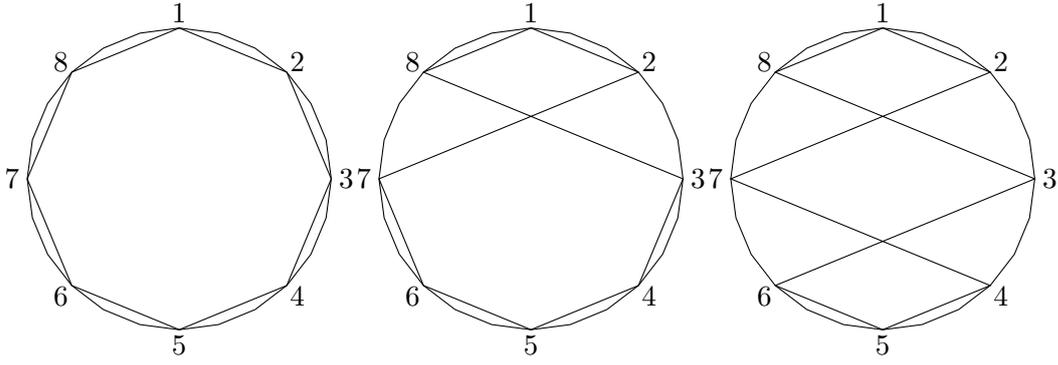
\begin{figure}[H]
\begin{minipage}{0.3\hsize}
\begin{center}
\begin{tikzpicture}
\draw[-=.5,domain=0:360] plot ({2*cos(\x)},{2*sin(\x)});
\coordinate (V1) at (90:2);
\coordinate (V2) at (45:2);
\coordinate (V3) at (0:2);
\coordinate (V4) at (-45:2);
\coordinate (V5) at (-90:2);
\coordinate (V6) at (-135:2);
\coordinate (V7) at (-180:2);
\coordinate (V8) at (-225:2);
\node at (90:2.2){$1$};
\node at (45:2.2){$2$};
\node at (0:2.2){$3$};
\node at (-45:2.2){$4$};
\node at (-90:2.2){$5$};
\node at (-135:2.2){$6$};
\node at (-180:2.2){$7$};
\node at (-225:2.2){$8$};
\draw (V1) -- (V2) -- (V3) -- (V4) -- (V5) -- (V6) -- (V7) -- (V8) -- cycle;
\end{tikzpicture}
\end{center}
\end{minipage}
\begin{minipage}{0.3\hsize}
\begin{center}
\begin{tikzpicture}
\draw[-=.5,domain=0:360] plot ({2*cos(\x)},{2*sin(\x)});
\coordinate (V1) at (90:2);
\coordinate (V2) at (45:2);
\coordinate (V3) at (0:2);
\coordinate (V4) at (-45:2);
\coordinate (V5) at (-90:2);
\coordinate (V6) at (-135:2);
\coordinate (V7) at (-180:2);
\coordinate (V8) at (-225:2);
\node at (90:2.2){$1$};
\node at (45:2.2){$2$};
\node at (0:2.2){$3$};
\node at (-45:2.2){$4$};
\node at (-90:2.2){$5$};
\node at (-135:2.2){$6$};
\node at (-180:2.2){$7$};
\node at (-225:2.2){$8$};
\draw (V1) -- (V2) -- (V7) -- (V6) -- (V5) -- (V4) -- (V3) -- (V8) -- cycle;
\end{tikzpicture}
\end{center}
\end{minipage}
\begin{minipage}{0.3\hsize}
\begin{center}
\begin{tikzpicture}
\draw[-=.5,domain=0:360] plot ({2*cos(\x)},{2*sin(\x)});
\coordinate (V1) at (90:2);
\coordinate (V2) at (45:2);
\coordinate (V3) at (0:2);
\coordinate (V4) at (-45:2);
\coordinate (V5) at (-90:2);
\coordinate (V6) at (-135:2);
\coordinate (V7) at (-180:2);
\coordinate (V8) at (-225:2);
\node at (90:2.2){$1$};
\node at (45:2.2){$2$};
\node at (0:2.2){$3$};
\node at (-45:2.2){$4$};
\node at (-90:2.2){$5$};
\node at (-135:2.2){$6$};
\node at (-180:2.2){$7$};
\node at (-225:2.2){$8$};
\draw (V1) -- (V2) -- (V7) -- (V4) -- (V5) -- (V6) -- (V3) -- (V8) -- cycle;
\end{tikzpicture}
\end{center}
\end{minipage}
\label{Hourglass}
\caption{Producing an hourglass with several sections}
\end{figure}

\end{proof}

\begin{rem}\label{rem:std}
We call $\alpha\in\mathfrak{S}_n$ a standard representative if $\alpha$ is obtained by the following process: If $n=3$, $\alpha=123$ is the unique standard representative. Suppose $n\geq 4$. We begin with a consecutive digits $12\cdots (n-1)$. First, we choose a consecutive digits $i(i+1)$ and revert them. We write $\alpha^\prime(1)\alpha^\prime(2)\cdots\alpha^\prime(n-1)$ for the resulting sequence. Then, we choose a consecutive digits $\alpha^\prime(i^\prime)\alpha^\prime(i^\prime+1)$ and revert it to $\alpha^\prime(i^\prime+1)\alpha^\prime(i^\prime)$ unless $(\alpha^\prime(i^\prime),\alpha^\prime(i^\prime+1))=(1,n-1)$. We repeat this process and we arrive at a sequence $\alpha(1)\cdots \alpha(n-1)$. A standard representative $\alpha\in\mathfrak{S}_n$ is then given by $\alpha(1)\cdots \alpha(n-1)n$. For example, if $n=4$, standard representatives are $1234, 1324, 2134$. Clearly, standard representatives give rise to a complete set of representatives of the quotient $\mathfrak{S}_n/D_n$. Under the identification $\mathcal{M}_{0,n}\simeq \{ (t_2,\dots,t_{n-2})\in\C^{n-3}\}\setminus\bigcup_{i=2}^{n-2}\{ t_i(t_i-1)=0\}\cup\bigcup_{2\leq i<j\leq n-2}\{ t_i=t_j\}$, each standard representative $\alpha\in\mathfrak{S}_{n}$ corresponds to a chamber $\Delta(\alpha)=\{ t_{\alpha(1)}<\cdots<t_{\alpha(n-1)}\}$ where we have set $t_1=0$ and $t_{n-1}=1$. 
\end{rem}

\begin{exa}\label{exa:3.3}
We take $\alpha=134256\in\mathfrak{S}_6$. CHY rule produces the graph as in Figure \ref{TheRule0}.

\begin{figure}[H]
\begin{center}
\begin{tikzpicture}
\draw[-=.5,domain=0:360] plot ({2*cos(\x)},{2*sin(\x)});
\coordinate (2) at ({2*cos(0)},{2*sin(0)});
\coordinate (1) at ({2*cos(60)},{2*sin(60)});
\coordinate (6) at ({2*cos(120)},{2*sin(120)});
\coordinate (5) at ({2*cos(180)},{2*sin(180)});
\coordinate (4) at ({2*cos(240)},{2*sin(240)});
\coordinate (3) at ({2*cos(280)},{2*sin(280)});
\node at (60:2.3){$1$};
\node at (0:2.3){$2$};
\node at (120:2.3){$6$};
\node at (180:2.3){$5$};
\node at (240:2.3){$4$};
\node at (280:2.3){$3$};
\draw[name path=s13] (1) -- (3);
\draw[name path=s34] (3) -- (4);
\draw[name path=s42] (4) -- (2);
\draw[name path=s25] (2) -- (5);
\draw[name path=s56] (5) -- (6);
\draw[name path=s61] (6) -- (1);
\path[name intersections={of= s13 and s42, by={A}}];
\path[name intersections={of= s13 and s25, by={B}}];
\path (A)++(0.3,0.6) coordinate (A1);
\path (A)++(-0.4,-0.8) coordinate (A2);
\coordinate (B1) at ($ (A1) !5! (B) $ );
\draw[red] (A1) -- ($ (A1) !1.1! (2) $ );
\draw[red] (A2) -- ($ (A2) !1.2! (3) $ );
\draw[red] (A2) -- ($ (A2) !1.1! (4) $ );
\draw[red] (B1) -- ($ (B1) !1.1! (1) $ );
\draw[red] (B1) -- ($ (B1) !1.1! (5) $ );
\draw[red] (B1) -- ($ (B1) !1.1! (6) $ );
\draw[red] (A1) -- (A2);
\draw[red] (A1) -- (B1);
\node at (A1){$\bullet$};
\node at (A2){$\bullet$};
\node at (B1){$\bullet$};
\end{tikzpicture}
\end{center}
\caption{The tree corresponding to the intersection $K(\mathbb{I}_6)\cap K(134256)$.}
\label{TheRule0}
\end{figure}
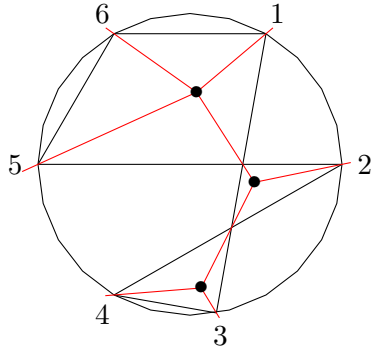

Let us take $\alpha=146325\in\mathfrak{S}_6$. In this case, the graph produced by the CHY rule depends on the configuration of the marked points $1,2,\dots,6$ (Figure \ref{TheRule}). However, it has a cycle in any case. This means that the intersection $K(\mathbb{I}_6)\cap K(146325)$ is empty in $\overline{\mathcal{M}}_{0,6}(\R)$.

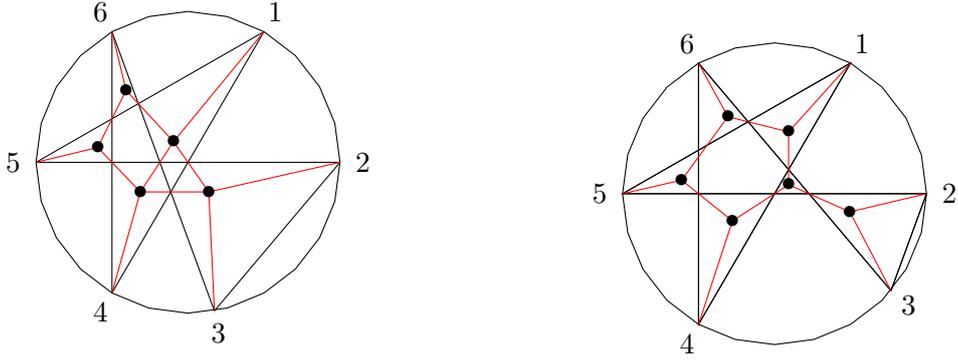
\begin{figure}[H]
\begin{minipage}{0.5\hsize}
\begin{center}
\begin{tikzpicture}
\draw[-=.5,domain=0:360] plot ({2*cos(\x)},{2*sin(\x)});
\coordinate (2) at ({2*cos(0)},{2*sin(0)});
\coordinate (1) at ({2*cos(60)},{2*sin(60)});
\coordinate (6) at ({2*cos(120)},{2*sin(120)});
\coordinate (5) at ({2*cos(180)},{2*sin(180)});
\coordinate (4) at ({2*cos(240)},{2*sin(240)});
\coordinate (3) at ({2*cos(280)},{2*sin(280)});
\node at (60:2.3){$1$};
\node at (0:2.3){$2$};
\node at (120:2.3){$6$};
\node at (180:2.3){$5$};
\node at (240:2.3){$4$};
\node at (280:2.3){$3$};
\draw[name path=s14] (1) -- (4);
\draw[name path=s46] (4) -- (6);
\draw[name path=s63] (6) -- (3);
\draw[name path=s32] (3) -- (2);
\draw[name path=s25] (2) -- (5);
\draw[name path=s51] (5) -- (1);
\path[name intersections={of= s14 and s63, by={A}}];
\path[name intersections={of= s14 and s25, by={B}}];
\path[name intersections={of= s46 and s25, by={C}}];
\path[name intersections={of= s46 and s51, by={D}}];
\path[name intersections={of= s51 and s63, by={E}}];
\path[name intersections={of= s25 and s63, by={F}}];
\path (A)++(0.5,0) coordinate (A1);
\draw[name path=L1,opacity=0] ($ (A1) !3! (A) $ ) -- ($ (A) !3! (A1) $);
\draw[name path=L2,opacity=0] ($ (A1) !3! (B) $ ) -- ($ (B) !3! (A1) $);
\path (A)++(-0.4,0) coordinate (A2);
\draw[name path=L3,opacity=0] ($ (A2) !3! (F) $ ) -- ($ (F) !3! (A2) $);
\path[name intersections={of= L2 and L3, by={B1}}];
\draw[name path=L4,opacity=0] ($ (B1) !3! (E) $ ) -- ($ (E) !3! (B1) $);
\draw[name path=L5,opacity=0] ($ (A2) !3! (C) $ ) -- ($ (C) !3! (A2) $);
\coordinate (E1) at ($ (B1) !1.35! (E) $);
\draw[name path=L6,opacity=0] ($ (E1) !3! (D) $ ) -- ($ (D) !3! (E1) $);
\path[name intersections={of= L5 and L6, by={C1}}];
\draw[red] (A1) -- (2);
\draw[red] (A1) -- (3);
\draw[red] (A1) -- (B1);
\draw[red] (A1) -- (A2);
\draw[red] (A2) -- (4);
\draw[red] (B1) -- (1);
\draw[red] (B1) -- (E1);
\draw[red] (E1) -- (6);
\draw[red] (C1) -- (E1);
\draw[red] (C1) -- (5);
\draw[red] (A2) -- (C1);
\draw[red] (A2) -- (B1);
\node at (A1){$\bullet$};
\node at (A2){$\bullet$};
\node at (B1){$\bullet$};
\node at (C1){$\bullet$};
\node at (E1){$\bullet$};
\end{tikzpicture}
\end{center}
\end{minipage}
\begin{minipage}{0.5\hsize}
\begin{center}
\begin{tikzpicture}
\draw[-=.5,domain=0:360] plot ({2*cos(\x)},{2*sin(\x)});
\coordinate (2) at ({2*cos(0)},{2*sin(0)});
\coordinate (1) at ({2*cos(60)},{2*sin(60)});
\coordinate (6) at ({2*cos(120)},{2*sin(120)});
\coordinate (5) at ({2*cos(180)},{2*sin(180)});
\coordinate (4) at ({2*cos(240)},{2*sin(240)});
\coordinate (3) at ({2*cos(320)},{2*sin(320)});
\node at (60:2.3){$1$};
\node at (0:2.3){$2$};
\node at (120:2.3){$6$};
\node at (180:2.3){$5$};
\node at (240:2.3){$4$};
\node at (320:2.3){$3$};
\draw (1) -- (4) -- (6) -- (3) -- (2) -- (5) -- cycle;
\draw[name path=s14] (1) -- (4);
\draw[name path=s46] (4) -- (6);
\draw[name path=s63] (6) -- (3);
\draw[name path=s32] (3) -- (2);
\draw[name path=s25] (2) -- (5);
\draw[name path=s51] (5) -- (1);
\path[name intersections={of= s14 and s63, by={A}}];
\path[name intersections={of= s14 and s25, by={B}}];
\path[name intersections={of= s46 and s25, by={C}}];
\path[name intersections={of= s46 and s51, by={D}}];
\path[name intersections={of= s51 and s63, by={E}}];
\path[name intersections={of= s25 and s63, by={F}}];
\path (A)++(0,-0.2) coordinate (A1);
\draw[name path=L1,opacity=0] ($ (A1) !5! (A) $ ) -- ($ (A) !5! (A1) $);
\draw[name path=L2,opacity=0] ($ (A1) !7! (B) $ ) -- ($ (B) !7! (A1) $);
\path (A)++(0,0.5) coordinate (A2);
\draw[name path=L3,opacity=0] ($ (A1) !5! (F) $ ) -- ($ (F) !5! (A1) $);
\draw[name path=L4,opacity=0] ($ (A2) !5! (E) $ ) -- ($ (E) !5! (A2) $);
\coordinate (B1) at ($ (A1) !4! (B) $ );
\draw[name path=L5,opacity=0] ($ (B1) !5! (C) $ ) -- ($ (C) !5! (B1) $);
\coordinate(E1) at ($ (A2) !1.5! (E) $ );
\draw[name path=L6,opacity=0] ($ (E1) !5! (D) $ ) -- ($ (D) !5! (E1) $);
\coordinate (F1) at ($ (A1) !3! (F) $ );
\coordinate (C1) at ($ (B1) !1.5! (C) $ );
\draw[red] (A1) -- (F1);
\draw[red] (2) -- (F1);
\draw[red] (3) -- (F1);
\draw[red] (A1) -- (B1);
\draw[red] (A2) -- (1);
\draw[red] (A2) -- (E1);
\draw[red] (6) -- (E1);
\draw[red] (B1) -- (C1);
\draw[red] (B1) -- (4);
\draw[red] (5) -- (C1);
\draw[red] (A1) -- (A2);
\draw[red] (E1) -- (C1);
\node at (A1){$\bullet$};
\node at (A2){$\bullet$};
\node at (B1){$\bullet$};
\node at (C1){$\bullet$};
\node at (E1){$\bullet$};
\node at (F1){$\bullet$};
\end{tikzpicture}
\end{center}
\end{minipage}
\caption{CHY rule may produce different graphs}
\label{TheRule}
\end{figure}

\end{exa}

Let us recall the definition of the relative winding number $w(\alpha|\beta)$ (\cite{M}). The rule is to first draw the permutation $\alpha$ on a circle in a clockwise order, and then follow the points according to the other permutation $\beta$ by always going clockwise. The relative winding number $w(\alpha|\beta)$, is then given by the total number of cycles completed. As an example, we have $w(\mathbb{I}_5|31425)=3$ as in Figure \ref{fig:RWN}.

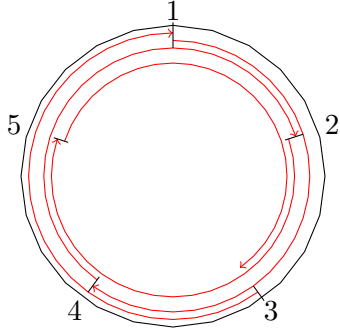
\begin{figure}[H]
\begin{center}
\begin{tikzpicture}
\coordinate (V1) at (90:2);
\coordinate (V2) at (18:2);
\coordinate (V3) at (-54:2);
\coordinate (V4) at (-126:2);
\coordinate (V5) at (162:2);
\draw[-=.5,domain=0:360] plot ({2*cos(\x)},{2*sin(\x)});
\node at (90:2.2){$1$};
\node at (18:2.2){$2$};
\node at (-54:2.2){$3$};
\node at (-126:2.2){$4$};
\node at (162:2.2){$5$};
\draw[red,->=.5,domain=-54:-270] plot ({1.9*cos(\x)},{1.9*sin(\x)});
\draw[red,->=.5,domain=90:-126] plot ({1.8*cos(\x)},{1.8*sin(\x)});
\draw[red,->=.5,domain=-126:-342] plot ({1.7*cos(\x)},{1.7*sin(\x)});
\draw[red,->=.5,domain=18:-198] plot ({1.6*cos(\x)},{1.6*sin(\x)});
\draw[red,->=.5,domain=162:-54] plot ({1.5*cos(\x)},{1.5*sin(\x)});
\draw[-=.5,domain=1.8:2.04] plot (-54:{\x});
\draw[-=.5,domain=1.7:2.04] plot (90:{\x});
\draw[-=.5,domain=1.65:1.9] plot (-126:{\x});
\draw[-=.5,domain=1.55:1.8] plot (18:{\x});
\draw[-=.5,domain=1.45:1.65] plot (162:{\x});
\end{tikzpicture}
\end{center}
\caption{The relative winding number $w(12345|31425)=3$}
\label{fig:RWN}
\end{figure}

\section{A tree expansion formula}\label{sec:4}
We consider complex parameters $s_{ij}$ for $1\leq i<j\leq n$. For any pair $i<j$, we set $s_{ji}:=s_{ij}$. We assume the following constraint on the complex parameters $s_{ij}$:
\begin{equation}\label{eqn:5}
\sum_{1\leq j\leq n, j\neq i}s_{ij}=0\ \ (i=1,\dots,n).
\end{equation}
We take $p\geq 3$ numbers $0\leq\alpha_1<\dots<\alpha_p< n$. We define consecutive digits $a_1,\dots,a_{p+1}$ by $a_1:=\alpha_1+1\cdots \alpha_2$, $a_2:=\alpha_2+1\cdots\alpha_3,\dots$, $a_{p}:=\alpha_p+1\cdots n 1\cdots \alpha_1$. We regard the associahedron $K_{p-1}$ as a set of brackets of digits $a_1a_2\cdots a_{p-1}$. In this sense, $K_{p-1}$ is also denoted by $K_{a_1\cdots a_{p-1}}$. For any bracket $e=(a_i\cdots a_j)$, we set $s_e:=s_{a_i\cdots a_j}=s_{\alpha_{i}+1\cdots\alpha_j}=\sum_{\alpha_{i}< k<l\leq\alpha_j}s_{kl}$. Since any face $F<K_{a_1\cdots a_{p-1}}$ can be reinterpreted as a tree diagram, we set
\begin{equation}
m(a_1,\dots, a_p):=\sum_{F<K_{a_1\cdots a_{p-1}}:admissible}\prod_{v\in V(F)_{int}}C_{\frac{|v|-3}{2}}\prod_{e\in E(F)_{int}}\cot\pi s_e.
\end{equation}
Here, $C_k$ denotes the $k$-th Catalan number. We also write $m(a_i,\dots, a_{p},a_1,\dots,a_{i-1})$ for $m(a_1,\dots, a_p)$. Note that $m(a_1,a_2,a_3)=1$ if $a_1a_2a_3=12\cdots n$.

Now suppose that a face $F<K_{n-1}=K_{12\cdots n-1}$ is given. For any internal vertex $v$ of $F$, we define the consecutive digits $a_1,\dots,a_{|v|}$ as follows: we consider all the edges $e_1,\dots,e_{|v|}$ containing $v$. If we remove an edge $e$ from $F$, $F\setminus e$ is decomposed into two connected components. We consider the component $C$ which does not contain this vertex $v$. The component $C$ has labeled vertices $i\cdots j$ which is a sequence of consecutive numbers. Here, we regard a sequence of the form $a a+1\cdots n 12\cdots b$ ($b<a$) as consecutive numbers. We set $a_e:=i\cdots j$. Without loss of generality, we may assume that $a_{e_1}\cdots a_{e_{|v|}}=1\cdots n$ in a circular sense, namely, we identify $1\cdots n$ with $23\cdots n1$ and so on. We set
\begin{equation}
m_v(F):=m(a_{e_1},\dots, a_{e_{|v|}}).
\end{equation}
If we use a sequence $\alpha(1)\cdots\alpha(n)$ instead of a sequence $12\cdots n$ for some $\alpha\in\mathfrak{S}_n$ and $F<K(\alpha)\simeq K_{\alpha(1)\cdots\alpha(n-1)}$, we write $m_v^{\alpha}(F)$. Note that if $e$ is an internal edge of $F<K(\alpha)$ and if $I$ is a connected component of the complement $F\setminus e$, we have 
\begin{equation}\label{eqn:se}
s_e=\sum_{i,j\in I,\ i<j}s_{ij}.
\end{equation}
Note that the formula (\ref{eqn:se}) does not depend on the choice of the connected component $I$ in view of (\ref{eqn:5}).

We set
\begin{equation}\label{eqn:MizeraFcn}
\Phi(t_1,\dots,t_n):=\prod_{1\leq i<j\leq n}(t_j-t_i)^{s_{ij}}.
\end{equation}
We write $\C\Phi$ for the local system on ${\rm Conf}_n(\mathbb{P}^1)$ whose local section is a determination of the multi-valued function $\Phi$. In view of (\ref{eqn:5}), $\C\Phi$ induces a local system on $\mathcal{M}_{0,n}$ which is still denoted by the same symbol. We write $\C\Phi^{-1}$ for the dual local system of $\C\Phi$. We are interested in computing the twisted homology intersection form at the middle dimension
\begin{equation}
\langle\bullet,\bullet\rangle_h:\Homo_{n-3}(\mathcal{M}_{0,n};\C\Phi)\times\Homo_{n-3}^{lf}(\mathcal{M}_{0,n};\C\Phi^{-1})\rightarrow\C.
\end{equation}
Here, the superscript $lf$ stands for the word ``locally finite'' and $\Homo_{*}^{lf}$ denotes the locally finite (or Borel-Moore) homology group. We say that the regularization condition is satisfied if the canonical morphism
\begin{equation}\label{eqn:can}
\Homo_k(\mathcal{M}_{0,n};\C\Phi)\rightarrow\Homo_k^{lf}(\mathcal{M}_{0,n};\C\Phi)
\end{equation}
is an isomorphism for any $k$. If the regularization condition is satisfied, both the homology group $\Homo_k(\mathcal{M}_{0,n};\C\Phi)$ and the locally finite homology group $\Homo_k^{lf}(\mathcal{M}_{0,n};\C\Phi)$ vanish unless $k=n-3$. The inverse map of the canonical morphism (\ref{eqn:can}) is denoted by $reg$. In order to justify the regularization condition, let us recall the following result.

\begin{prop}[\cite{CDO}]\label{prop:CDO}
Let $\mathscr{A}=\{ H\}$ be a hyperplane arrangement in $\C^n$ and let $\mathscr{A}_\infty:=\{ H\}\cup H_\infty$ be its associated projective arrangement where $H_\infty$ is the hyperplane at infinity. Let $l_H$ be linear forms defining $H$. We consider a local system $\mathcal{L}=\C\prod_{H\in\mathscr{A}}l_H^{\alpha_H}$ on the complement $X:=\C^n\setminus \mathscr{A}$ for some $\alpha_H\in\C$. We set $\alpha_{H_\infty}:=-\sum_{H\in\mathscr{A}}\alpha_H$. 

If for any dense edge $E\in D(\mathscr{A}_\infty)$, the condition $\alpha_E:=\sum_{E\subset H\in\mathscr{A}_\infty}\alpha_H\notin \Z$ holds, then one has a canonical isomorphism
\begin{equation}
\Homo_k(X;\mathcal{L})\simeq\Homo^{lf}_k(X;\mathcal{L})
\end{equation}
for any integer $k$.
\end{prop}

\noindent
To be more precise, Proposition \ref{prop:CDO} follows from \cite[lemma 3]{CDO} combined with the composition theorem of derived functors. We assume the following condition:

\vspace{.5em}

\noindent
\underline{Condition ($*$)}
\begin{quote}
For any element $[\alpha]\in\mathfrak{S}_n/D_n$, any face $F<K(\alpha)$ and for any internal edge $e$ of $F$, the complex number $s_e$ is not an integer.
\end{quote}
In view of the fact that $\overline{\mathcal{M}}_{0,n}$ can be realized as an iterated blowing-up of $\mathbb{P}^{n-3}$ along dense edges of a hyperplane arrangement (\cite[\S 4]{Devadoss}, \cite[\S4]{Kap}, \cite[Chapter 4]{Kap2}) and each edge $e$ corresponds to a dense edge at which the sum of relevant exponents $s_{ij}$ is precisely $s_e$, the regularization condition is satisfied under the condition ($*$). Namely, both $\Homo_k(\mathcal{M}_{0,n};\C\Phi^\pm)$ and $\Homo_k^{lf}(\mathcal{M}_{0,n};\C\Phi^\pm)$ are zero unless $k=n-3$ and we have a natural isomorphism $\Homo_{n-3}(\mathcal{M}_{0,n};\C\Phi^\pm)\simeq\Homo_{n-3}^{lf}(\mathcal{M}_{0,n};\C\Phi^\pm)$.

Now let us recall the orientable double cover $\overline{\mathcal{M}}_{0,n}^{or}(\R)$ of the real moduli space $\overline{\mathcal{M}}_{0,n}(\R)$ constructed in \cite{DM}. Let $\pi:\overline{\mathcal{M}}_{0,n}^{or}(\R)\rightarrow\overline{\mathcal{M}}_{0,n}(\R)$ be the covering map of loc. cit. and let us fix an orientation of $\overline{\mathcal{M}}_{0,n}^{or}(\R)$. In loc. cit., the authors fix one point $t_n=\infty$ which amounts to the natural bijection $\mathfrak{S}_{n-1}/\langle (n-1) (n-2) \cdots 1\rangle\simeq\mathfrak{S}_{n}/D_n$. Note that $\mathfrak{S}_{n-1}$ can be identified with a subgroup of $\mathfrak{S}_{n}$ consisting of permutations $\alpha$ such that $\alpha(n)=n$.  By construction, the preimage $\pi^{-1}\left( \mathcal{M}_{0,n}(\R)\right)$ is a disjoint union of $(n-1)!$ copies of the associahedron $K_{n-1}$ labeled by the elements of $\mathfrak{S}_{n-1}$. Let the symbol $C(\alpha)$ denote the associahedron in $\pi^{-1}\left( \mathcal{M}_{0,n}(\R)\right)$ labeled by an element $\alpha\in\mathfrak{S}_{n-1}\simeq\mathfrak{S}_n/\langle 23\cdots n1\rangle$. We have $\pi(C(\alpha))=\Delta(\alpha)$. The orientation of $C(\alpha)$ is naturally induced from that of $\overline{\mathcal{M}}_{0,n}^{or}(\R)$. By abuse of notation, we write $C(\alpha)$ for the image of $C(\alpha)$ through the morphism $\pi$ in $\mathcal{M}_{0,n}(\R)$. We choose the {\it standard loading} of the multivalued function $\Phi$ on $C(\alpha)$, that is, we choose the branch of $\Phi$ so that we have $\Phi>0$ on $C(\alpha)$ when all the parameters $s_{ij}$ are real. With this choice of a branch of $\Phi$, $C(\alpha)$ defines an element of the locally finite homology group $\Homo_{n-3}^{lf}(\mathcal{M}_{0,n};\C\Phi)$ which is denoted by $[C^+(\alpha)]$. The same argument defines a homology class $[C^-(\alpha)]$ of $\Homo_{n-3}^{lf}(\mathcal{M}_{0,n};\C\Phi^{-1})$ for any $\alpha\in\mathfrak{S}_{n-1}\simeq\mathfrak{S}_{n}/\langle 23\cdots n1\rangle$. In view of \cite{Kohno}, we see that the twisted homology group $\Homo^{lf}_{n-3}(\mathcal{M}_{0,n};\C\Phi^\pm)$ is generated by $\{[C^\pm(\alpha)]\}_{\alpha\in\mathfrak{S}_{n-1}}$. As a basis of $\Homo^{lf}_{n-3}(\mathcal{M}_{0,n};\C\Phi^\pm)$, one can take, for example, $\{[C^\pm(\alpha)]\}_{\alpha\in\mathfrak{S}_{n-3}}$ where $\mathfrak{S}_{n-3}$ is identified with the set of permutations which fix $1,n-1$ and $n$. This is a basis consisting of bounded chambers. For the cycles $[C^{\pm}(\mathbb{I}_n)]\in\Homo_{n-3}^{lf}\left( \mathcal{M}_{0,n};\C\Phi^\pm\right)$, we have the following formula.
\begin{thm}\label{thm:1}
We assume the condition ($*$). Then, one has a formula
\begin{align}
&\langle reg[C^+(\mathbb{I}_n)],[C^-(\mathbb{I}_n)]\rangle_h\nonumber\\
=&\left(\frac{\ii}{2}\right)^{n-3}m(1,2,\dots,n)\\
=&\left(\frac{\ii}{2}\right)^{n-3}\sum_{F<K_{n-1}:\text{admissible}}\prod_{v\in V(F)_{\rm int}}C_{\frac{|v|-3}{2}}\prod_{e\in E(F)_{\rm int}}\cot(\pi s_e).
\end{align}
\end{thm}

\noindent
More generally, we have a 
\begin{thm}\label{thm:OffDiagonal}
We assume the condition ($*$). For any $[\alpha],[\beta]\in\mathfrak{S}_n/\langle 23\cdots n1\rangle$ with a non-empty intersection\footnote{If $K(\alpha)\cap K(\beta)=\varnothing$ in $\overline{\mathcal{M}}_{0,n}(\R)$, one has $\langle reg[C^+(\alpha)],[C^-(\beta)]\rangle_h=0$ by the definition of the twisted homology intersection number.} $K(\alpha)\cap K(\beta)=F$ in $\overline{\mathcal{M}}_{0,n}(\R)$, one has a formula
\begin{equation}\label{eqn:OffD}
\langle reg[C^+(\alpha)],[C^-(\beta)]\rangle_h=(-1)^{w(\alpha|\beta)+1}\left(\frac{\ii}{2}\right)^{n-3}\prod_{e\in E(F)_{int}}\csc(\pi s_e)\prod_{v\in V(F)_{int}}m_v^\alpha(F).
\end{equation}

\end{thm}
Note that when $\alpha=\beta$, $F$ is equal to $K(\alpha)$ and therefore, $E(F)_{int}=\varnothing$ and $V(F)_{int}$ is a single point. This case is reduced to Theorem \ref{thm:1}. In the following, we simply write $\langle [C^+(\alpha)],[C^-(\beta)]\rangle_h$ for the homology intersection number $\langle reg[C^+(\alpha)],[C^-(\beta)]\rangle_h$.

\section{Proof of Theorem \ref{thm:1}}

\begin{proof}
By an induction on the natural number $n$, we can prove an identity
\begin{equation}\label{lem:1}
\frac{1}{(T_1-1)\cdots(T_n-1)}=\frac{1}{2^n}\sum_{k=0}^n(-1)^{n-k}\sum_{I\subset\{ 1,\dots,n\},|I|=k}\prod_{i\in I}\frac{T_i+1}{T_i-1}.
\end{equation}
We set $e(\alpha):=e^{2\pi\ii\alpha}$ and $t(\alpha):=\frac{e(\alpha)+1}{e(\alpha)-1}$. For a vector $s=(s_1,s_2,\dots)$, we set $t(s):=t(s_1)t(s_2)\cdots$. For any bracket $a:=(i\cdots j)$, we set $s_a:=s_{i\dots j}:=\sum_{i\leq k<l\leq j}s_{kl}$. Since any face $F<K_{n-1}$ is a set of brackets, we set $s_F:=(s_a)_{a\in F}$. Let us recall the construction of $\overline{\mathcal{M}}_{0,n}$ as an iterated blowing-up of $\mathbb{P}^{n-3}$ (\cite[\S 4]{Devadoss}, \cite[\S4]{Kap}, \cite[Chapter 4]{Kap2}). Any bracket $a=i\cdots j$ in $12\cdots n-1$ corresponds to the proper transform of the linear subvariety $\{ t_i=\cdots =t_j\}$ in $\overline{\mathcal{M}}_{0,n}$ around which the local system $\C\Phi^{\pm}$ has the eigenvalue of local monodromy $e(\pm s_a)$. Therefore, as in \cite[Lemma 1]{MY}, we obtain
\begin{equation}\label{eqn:2}
\langle[C^+(\mathbb{I}_n)],[C^-(\mathbb{I}_n)]\rangle_h=(-1)^{n-3}\sum_{F<K_{n-1}}\prod_{a\in F}\frac{1}{e(s_a)-1},
\end{equation}
which can be deduced from \cite[p175, Proposition]{KY}. On the other hand, we obtain a formula
\begin{equation}\label{eqn:1}
\prod_{a\in F}\frac{1}{e(s_a)-1}=\frac{1}{2^{\codim F}}\sum_{F<G}(-1)^{\dim G-\dim F}t(s_G)
\end{equation}
in view of (\ref{lem:1}). Comparing (\ref{eqn:2}) and (\ref{eqn:1}), we can expand the intersection number $\langle[C^+(\mathbb{I}_n)],[C^-(\mathbb{I}_n)]\rangle_h$ into a sum of $t(s_F)$'s. The coefficient of $t(s_F)$ is given by
\begin{equation}
(-1)^{n-3}\sum_{F^\prime<F}\frac{1}{2^{\codim F^\prime}}(-1)^{\dim F-\dim F^\prime}=\frac{(-1)^{\codim F}}{2^{n-3}}\sum_{F^\prime< F}(-2)^{\dim F^\prime}.
\end{equation}
We set $C_F:=\sum_{F^\prime< F}(-2)^{\dim F^\prime}$. In view of Proposition \ref{prop:2.3}, any face $F^\prime$ of $F$ is decomposed into a product of faces $f_v$ of $K_{|v|-1}$ as $F^\prime=\prod_{v\in V(F)_{int}}f_v$. Thus, we obtain
\begin{equation}
C_F=\sum_{\prod_vf_v<\prod_vK_{|v|-1}}(-2)^{\sum_v\dim f_v}=\prod_{v\in V(F)_{int}}\left( \sum_{f_v<K_{|v|-1}}(-2)^{\dim f_v}\right)=\prod_{v\in V(F)_{int}}C_{K_{|v|-1}}.
\end{equation}
Therefore, we are reduced to computing $C_{K_{n-1}}$. Since the number of $k$-codimensional faces of $K_{n-1}$ is $\frac{1}{n-1}\binom{n-3}{k}\binom{n+k-1}{k+1}$ (\cite[LEMMA 3.2.1]{Devadoss}), we have a formula
\begin{equation}
C_{K_{n-1}}=\frac{1}{n-1}\sum_{k=0}^{n-3}(-2)^{n-3-k}\binom{n-3}{k}\binom{n+k-1}{k+1}.
\end{equation}
In view of Theorem \ref{prop:WZ}, we see that $C_F$ is zero unless $F$ is admissible. When $F$ is admissible, we obtain Theorem \ref{thm:1} again from Theorem \ref{prop:WZ}.
\end{proof}

\section{Proof of Theorem \ref{thm:OffDiagonal}}
Before going into the proof, let us first discuss the signature effect of blowing-up on twisted homology intersection numbers. Let $l_1,\dots,l_q$ be non-constant linear polynomials in $r$-variables $x=(x_1,\dots,x_r)$ with real coefficients. We assume that $l_1,\dots,l_p$ do not have a constant term and $\cup_{i=1}^p\{l_i=0\}$ is normal crossing. We consider domains $D_1,D_2\subset \R^r$ specified by the following relations:
\begin{align*}
\bar{D}_1:l_1\geq 0,\dots,l_p\geq 0;&\ l_{p+1}\geq 0,\dots,l_q\geq 0\\
\bar{D}_2:l_1\leq 0,\dots,l_p\leq 0;&\ l_{p+1}\geq 0,\dots,l_q\geq 0\\
\bar{D}_1\cap\bar{D}_2:l_1= 0,\dots,l_p= 0;&\ l_{p+1}\geq 0,\dots,l_q\geq 0
\end{align*}
for some integer $1\leq p\leq r$. Here, $\bar{D}_i$ denotes the closure of $D_i$. It is important to observe that we can equip $D_1$ and $D_2$ with an orientation induced from that of $\R^r$. After blowing-up, the orientation may differ. Let us see how the orientation changes. Since it is a local problem, we may assume that $l_1=x_1,\dots,l_p=x_p$, $\bar{D}_1=\{ x_1\geq 0,\dots,x_p\geq 0\}$ and $\bar{D}_2=\{ x_1\leq 0,\dots,x_p\leq 0\}$. We consider the blowing-up $\pi:X\rightarrow\C^r$ of $\C^r$ along $\{ x_1=\cdots=x_p=0\}$. Let us take $(w_1:=x_1,w_2:=\frac{x_2}{x_1},\dots,w_p:=\frac{x_p}{x_1},x_{p+1},\dots,x_r)$ as a local coordinate of $X$. On $\R^r$, $D_1$ and $D_2$ are oriented in such a way that the $r$-form $\omega:=dx_1\wedge\cdots\wedge dx_r$ is positive. Let us observe that $\pi^*\omega=w_1^{p-1}dw_1\wedge\cdots\wedge dw_p\wedge dx_{p+1}\wedge\cdots\wedge dx_r$. The proper transforms of $D_1$ and $D_2$ are locally given by the equations
\begin{align}
\pi^{-1}(D_1)=\{ w_1>0,\dots,w_p>0\}\\
\pi^{-1}(D_2)=\{ w_1<0,w_2>0,\dots,w_p>0\}.
\end{align}
We can equip $\pi^{-1}(D_i)$ ($i=1,2$) with an orientation $dw_1\wedge\cdots\wedge dw_p\wedge dx_{p+1}\wedge\cdots\wedge dx_r>0$. Then, $\pi:\pi^{-1}(D_1)\tilde{\rightarrow}D_1$ is orientation-preserving for any $p$. However, $\pi:\pi^{-1}(D_1)\tilde{\rightarrow}D_1$ is orientation-preserving if and only if $p$ is odd. Therefore, even-codimensional blowing-up add a signature $(-)$ to the formula \cite[p175, Proposition]{KY} of homology intersection numbers.

Now, we prove Theorem \ref{thm:OffDiagonal}. Observe that it is enough to evaluate $\langle [C^+(\mathbb{I}_n)],[C^-(\alpha)]\rangle_h$. Let us suppose that the intersection $F=K(\mathbb{I}_n)\cap K(\alpha)$ in $\overline{\mathcal{M}}_{0,n}(\R)$ is non-empty and $\alpha$ is a standard representative in the sense of Remark \ref{rem:std}. The formula \cite[p175, Proposition]{KY}\footnote{This is a formula of a homology intersection number on a complement of a hyperplane arrangement in a projective space. However, since the computation of intersection is a local problem, we can apply the formula even after blowing-up the projective space. The signature effect of blowing-up must be taken into account.} in our setting reads
\begin{equation}\label{eqn:19}
\langle [C^+(\mathbb{I}_n)],[C^-(\alpha)]\rangle_h=\prod_{a\in F}(-1)^{|a|-2}\left(\frac{1}{2\ii}\right)^{|F|}\prod_{a\in F}\csc(\pi s_a)\times (-1)^{\dim F}\sum_{F^\prime<F}\prod_{a\in F^\prime}\frac{1}{e(s_a)-1}.
\end{equation}
Note that each bracket $a=(i\cdots j)\in F$ corresponds to a $(|a|-1)$-codimensional blowing-up along $\{ t_i=\cdots= t_j\}$. We have $\prod_{a\in F}(-1)^{|a|-2}\times (-1)^{|F|}=\prod_{a\in F}(-1)^{|a|-1}=(-1)^{w(\mathbb{I}_n|\alpha)+1}$. In view of Proposition \ref{prop:2.3} and the proof of Theorem \ref{thm:1}, we have an equality

\begin{equation}\label{eqn:20}
(-1)^{\dim F}\sum_{F^\prime<F}\prod_{a\in F^\prime}\frac{1}{e(s_a)-1}=\left(\frac{\ii}{2}\right)^{\dim F}\prod_{v\in V(F)_{int}}m_v(F).
\end{equation}
Combining (\ref{eqn:19}) and (\ref{eqn:20}), we obtain the desired formula (\ref{eqn:OffD}). Lastly, if $\alpha$ is not a standard representative, we see that $\alpha^\prime:=\alpha\circ (n (n-1)\cdots 1)\circ (23\cdots n1)^l$ for some integer $l$ is standard. Combining this observation with $(-1)^{w(\mathbb{I}_n|\alpha^\prime)}=(-1)^{w(\mathbb{I}_n|\alpha)+n}$ and $[C^-(\alpha^\prime)]=(-1)^n[C^-(\alpha)]$, we obtain (\ref{eqn:OffD}).

\section{A combinatorial identity}
\begin{thm}\label{prop:WZ}
For any non-negative integer $p$, we have
\begin{equation}\label{eqn:the sum}
\frac{1}{p+2}\sum_{k=0}^p(-2)^{p-k}\binom{p}{k}\binom{p+k+2}{k+1}=
\begin{cases}
(-1)^{\frac{p}{2}}C_{\frac{p}{2}}&(p: even)\\
0&(p:odd)
\end{cases}.
\end{equation}
\end{thm}

We give two different proofs of Theorem \ref{prop:WZ}. One is based on the Wilf-Zeilberger method (\cite{PWZ},\cite{WZ}), the other is based on the method of generating function. We thank Genki Shibukawa for sharing the second proof.

\vspace{.5em}

\noindent
({\it The first proof}) We set $F(p,k):=\frac{1}{p+2}(-2)^{p-k}\binom{p}{k}\binom{p+k+2}{k+1}$. We set $\Delta_p\cdot F(p,k):=F(p+1,k)$, $R(p,k):=\frac{-8(p+1)(2p+5)k(k+1)}{(p+3)(p-k+1)(p-k+2)}$, $G(p,k):=R(p,k)F(p,k)$. By a direct computation, we obtain
\begin{equation}\label{eqn:ct}
[(p+4)\Delta_p^2+4(p+1)]F(p,k)=G(p,k+1)-G(p,k).
\end{equation}
We set $f(p):=\sum_{k=0}^pF(p,k)=\sum_{k=0}^\infty F(p,k)$. Since $F(p,k)=G(p,k)=0$ if $k>p$, taking a summation of (\ref{eqn:ct}) gives rise to a difference equation
\begin{equation}\label{eqn:de}
(p+4)f(p+2)+4(p+1)f(p)=0.
\end{equation}
On the other hand, we set
\begin{equation}
g(p):=
\begin{cases}
(-1)^{\frac{p}{2}}C_{\frac{p}{2}}&(p:even)\\
0&(p:odd)
\end{cases}.
\end{equation}
It is easy to check that $g(p)$ is also a solution of (\ref{eqn:de}). Since $f(0)=g(0)$ and $f(1)=g(1)$, we obtain $f(p)=g(p)$ for any positive integer $p$.
\qed

\vspace{1em}

\noindent
({\it The second proof}) 
We first observe that the left-hand side of (\ref{eqn:the sum}) is same as the sum
\begin{equation}\label{eqn:Meixner}
\frac{1}{p+1}\sum_{k=0}^p(-2)^{p-k}\binom{p+1}{k+1}\binom{p+k+2}{k},
\end{equation}
which in turn equals to $\frac{1}{p+1}[u^p](1-u)^{-2}\left(\frac{1-2u}{1-u}\right)^{p+1}$. Here, the symbol $[u^p]f(u)$ denotes the coefficient of $u^p$ in the Taylor series expansion of $f(u)$. We apply the so-called Lagrange-B\"urmann formula (\cite[p.129]{WW}). In our setting, it is convenient to state it as follows: Let $H$ and $\phi$ be holomorphic functions defined at the origin. We assume that $\phi(0)\neq 0$. Let $g(z)$ be the inverse function of $\frac{u}{\phi(u)}$. Then, one has an equality 
\begin{equation}\label{eqn:LB}
[z^{p+1}]H(g(z))=\frac{1}{p+1}[u^p]H^\prime(u)\phi(u)^{p+1}
\end{equation}
for any non-negative integer $p$. In our case, we set $\phi(u)=\frac{1-2u}{1-u}$, $H(u)=\frac{1}{1-u}$ and $g(z)=\frac{1+2z-\sqrt{1+4z^2}}{2}$. The left-hand side of (\ref{eqn:LB}) is nothing but the sum (\ref{eqn:Meixner}). A simple computation shows that the right-hand side of (\ref{eqn:LB}) is that of (\ref{eqn:the sum}). Note that the generating function $\sum_{k=0}^\infty C_kz^k$ of Catalan numbers is given by $\frac{1-\sqrt{1-4z}}{2z}$. \qed

\begin{rem}
Let $M_{p}(x;\beta,\gamma)$ be the $p$-th Meixner polynomial (\cite[p.346]{AAR}). (\ref{eqn:Meixner}) is nothing but the special value $M_{p+1}(p;2,\frac{1}{2})$.
\end{rem}

\section{An illustrative example}
Let us compute the intersection number $\langle [C^+(\mathbb{I}_6)],[C^-(134256)]\rangle_h$. The intersection $K(\mathbb{I}_6)\cap K(134256)$ is computed in Example \ref{exa:3.3}. We have $w(\mathbb{I}_6|134256)=2$. Therefore, formula (\ref{eqn:OffD}) reads

\begin{align}
&\langle [C^+(\mathbb{I}_6)],[C^-(134256)]\rangle_h \nonumber\\
=&
\begin{minipage}{4truecm}
\begin{center}
\begin{tikzpicture}[scale=.5]
\draw[-=.5,domain=0:360] plot ({2*cos(\x)},{2*sin(\x)});
\coordinate (2) at ({2*cos(0)},{2*sin(0)});
\coordinate (1) at ({2*cos(60)},{2*sin(60)});
\coordinate (6) at ({2*cos(120)},{2*sin(120)});
\coordinate (5) at ({2*cos(180)},{2*sin(180)});
\coordinate (4) at ({2*cos(240)},{2*sin(240)});
\coordinate (3) at ({2*cos(280)},{2*sin(280)});
\node at (60:2.3){$1$};
\node at (0:2.3){$2$};
\node at (120:2.3){$6$};
\node at (180:2.3){$5$};
\node at (240:2.3){$4$};
\node at (280:2.3){$3$};
\draw[name path=s13] (1) -- (3);
\draw[name path=s34] (3) -- (4);
\draw[name path=s42] (4) -- (2);
\draw[name path=s25] (2) -- (5);
\draw[name path=s56] (5) -- (6);
\draw[name path=s61] (6) -- (1);
\path[name intersections={of= s13 and s42, by={A}}];
\path[name intersections={of= s13 and s25, by={B}}];
\path (A)++(0.3,0.6) coordinate (A1);
\path (A)++(-0.4,-0.8) coordinate (A2);
\coordinate (B1) at ($ (A1) !5! (B) $ );
\draw[red] (A1) -- ($ (A1) !1.1! (2) $ );
\draw[red] (A2) -- ($ (A2) !1.2! (3) $ );
\draw[red] (A2) -- ($ (A2) !1.1! (4) $ );
\draw[red] (B1) -- ($ (B1) !1.1! (1) $ );
\draw[red] (B1) -- ($ (B1) !1.1! (5) $ );
\draw[red] (B1) -- ($ (B1) !1.1! (6) $ );
\draw[red] (A1) -- (A2);
\draw[red] (A1) -- (B1);
\node at (A1){$\bullet$};
\node at (A2){$\bullet$};
\node at (B1){$\bullet$};
\end{tikzpicture}
\end{center}
\end{minipage}
\\
=&-\left(\frac{\ii}{2}\right)^3\csc(\pi s_{34})\csc(\pi s_{234})\times\nonumber\\
&
\begin{minipage}{4truecm}
\begin{center}
\begin{tikzpicture}[scale=.4]
\draw[-=.5,domain=0:360] plot ({2*cos(\x)},{2*sin(\x)});
\draw[red] (0,0) -- (45:2.2);
\draw[red] (0,0) -- (135:2.2);
\draw[red] (0,0) -- (-45:2.2);
\draw[red] (0,0) -- (-135:2.2);
\node at (45:2.7){$1$};
\node at (135:2.7){$6$};
\node at (-135:2.7){$5$};
\node at (-45:2.7){$234$};
\node at (0,0){$\bullet$};
\end{tikzpicture}
\end{center}
\end{minipage}
\times 
\begin{minipage}{4truecm}
\begin{center}
\begin{tikzpicture}[scale=.4]
\draw[-=.5,domain=0:360] plot ({2*cos(\x)},{2*sin(\x)});
\draw[red] (0,0) -- (90:2.2);
\draw[red] (0,0) -- (-30:2.2);
\draw[red] (0,0) -- (210:2.2);
\node at (90:2.7){$561$};
\node at (-30:2.7){$2$};
\node at (210:2.7){$34$};
\node at (0,0){$\bullet$};
\end{tikzpicture}
\end{center}
\end{minipage}
\times
\begin{minipage}{4truecm}
\begin{center}
\begin{tikzpicture}[scale=.4]
\draw[-=.5,domain=0:360] plot ({2*cos(\x)},{2*sin(\x)});
\draw[red] (0,0) -- (90:2.2);
\draw[red] (0,0) -- (-30:2.2);
\draw[red] (0,0) -- (210:2.2);
\node at (90:2.7){$5612$};
\node at (-30:2.7){$3$};
\node at (210:2.7){$4$};
\node at (0,0){$\bullet$};
\end{tikzpicture}
\end{center}
\end{minipage}
\\
=&-\left(\frac{\ii}{2}\right)^3\csc(\pi s_{34})\csc(\pi s_{234}) m(1,234,5,6)m(2,34,561)m(3,4,5612)
\end{align}
Both $m(2,34,561)$ and $m(3,4,5612)$ are a summation over one point $K_2$ and they are both equal to $1$. On the other hand, admissible trees in $K_3=K_{1,234,5}$ are precisely vertices and we have
\begin{align}
m(1,234,5,6) &=
\begin{minipage}{4truecm}
\begin{center}
\begin{tikzpicture}[scale=.4]
\draw[-=.5,domain=0:360] plot ({2*cos(\x)},{2*sin(\x)});
\draw[red] (0,0) -- (45:2.2);
\draw[red] (0,0) -- (135:2.2);
\draw[red] (0,0) -- (-45:2.2);
\draw[red] (0,0) -- (-135:2.2);
\node at (45:2.7){$1$};
\node at (135:2.7){$6$};
\node at (-135:2.7){$5$};
\node at (-45:2.7){$234$};
\node at (0,0){$\bullet$};
\end{tikzpicture}
\end{center}
\end{minipage}
\\
&=
\begin{minipage}{4truecm}
\begin{center}
\begin{tikzpicture}[scale=.4]
\draw[red] (0,0.5) -- (45:2.2);
\draw[red] (0,0.5) -- (135:2.2);
\draw[red] (0,-0.5) -- (-45:2.2);
\draw[red] (0,-0.5) -- (-135:2.2);
\draw[red] (0,-0.5) -- (0,0.5);
\node at (45:2.7){$1$};
\node at (135:2.7){$6$};
\node at (-135:2.7){$5$};
\node at (-45:2.7){$234$};
\node at (0,0.5){$\bullet$};
\node at (0,-0.5){$\bullet$};
\end{tikzpicture}
\end{center}
\end{minipage}
+
\begin{minipage}{4truecm}
\begin{center}
\begin{tikzpicture}[scale=.4]
\draw[red] (0.5,0) -- (45:2.2);
\draw[red] (-0.5,0) -- (135:2.2);
\draw[red] (0.5,0) -- (-45:2.2);
\draw[red] (-0.5,0) -- (-135:2.2);
\draw[red] (-0.5,0) -- (0.5,0);
\node at (45:2.7){$1$};
\node at (135:2.7){$6$};
\node at (-135:2.7){$5$};
\node at (-45:2.7){$234$};
\node at (0.5,0){$\bullet$};
\node at (-0.5,0){$\bullet$};
\end{tikzpicture}
\end{center}
\end{minipage}
\\
&=\cot(\pi s_{2345})+\cot(\pi s_{1234}).
\end{align}
The interested readers can find more examples in \cite{M} and \cite{MI}.

\end{document}